\documentclass{amsart}

\usepackage{amssymb}
\usepackage{graphicx, comment}
\usepackage{MnSymbol}

\usepackage{amsmath,amscd}

\usepackage{amsthm}

\title[Preservation of automatic continuity]{On preservation of automatic continuity}

\theoremstyle{definition}\newtheorem{theorem}{Theorem}
\theoremstyle{definition}\newtheorem*{A}{Theorem \ref{braidslender}}
\theoremstyle{definition}

\theoremstyle{definition}\newtheorem{bigtheorem}{Theorem}

\numberwithin{theorem}{section}
\theoremstyle{definition}\newtheorem{corollary}[theorem]{Corollary}
\theoremstyle{definition}
\theoremstyle{definition}\newtheorem{definition}[theorem]{Definition}
\theoremstyle{definition}\newtheorem{question}[theorem]{Question}
\theoremstyle{definition}
\theoremstyle{definition}\newtheorem{remark}[theorem]{Remark}
\theoremstyle{definition}
\theoremstyle{definition}\newtheorem{lemma}[theorem]{Lemma}
\theoremstyle{definition}
\theoremstyle{definition}
\theoremstyle{definition}
\theoremstyle{definition}\newtheorem{obs}[theorem]{Observation}

\newcommand{\W}{\mathcal{W}}
\newcommand{\HEG}{\operatorname{HEG}}

\newcommand{\Roots}{\operatorname{Roots}}

\begin{document}

\author[Samuel M. Corson]{Samuel M. Corson}
\address{Matematika Saila, UPV/EHU, Sarriena S/N, 48940, Leioa - Bizkaia, Spain}
\email{sammyc973@gmail.com}

\author[Ilya Kazachkov]{Ilya Kazachkov}
\address{Matematika Saila, UPV/EHU, Sarriena S/N, 48940, Leioa - Bizkaia, Spain}
\email{ilya.kazachkov@gmail.com}

\thanks{The authors are supported by ERC grant PCG-336983, Basque Government Grant IT974-16 and Spanish Government grant MTM2017-86802-P.}

\begin{abstract} 
	A group $G$ is called automatically continuous if  any homomorphism from a completely metrizable or locally compact Hausdorff group to $G$ has open kernel. In this paper, we study preservation of automatic continuity under group-theoretic constructions, focusing mainly on groups of size less than continuum. In particular, we consider group extensions and graph products. As a consequence, we establish automatic continuity of virtually poly-free groups, and hence of non-exceptional spherical Artin groups. 
	
	On the other hand, we show that if $G$ is automatically continuous, then so is any finitely generated residually $G$ group, hence, for instance, all finitely generated residually free groups are automatically continuous.
\end{abstract}

\maketitle

\begin{section}{Introduction}

The aim of this paper is to add to the increasing number of examples of automatically continuous groups. In this direction, on the one hand, we show that a number of prominent (classes of) groups are automatically continuous and, on the other hand, we establish several closure results for the class, notably, we show that it is closed under residual properties. In order to prove our results, we find new conditions that guarantee automatic continuity. An interesting feature of these conditions is that they are preserved under many group operations, such as extensions, graph products etc.

Define a group $G$ to be \emph{completely metrizable slender} (abbreviated \emph{cm-slender}) if every abstract homomorphism from a completely metrizable topological group to $G$ has open kernel.  Similarly define $G$ to be \emph{locally compact Hausdorff slender} (or \emph{lcH-slender}) if every abstract group homomorphism from a locally compact Hausdorff group to $G$ has open kernel.  Finally, $G$ is \emph{noncommutatively slender} (or \emph{n-slender}) if every abstract homomorphism from the Hawaiian earring group $\HEG$ to $G$ factors through projection to a canonical finite rank free subgroup (see Definition \ref{nslender}).  These notions imply a strong sort of automatic continuity- having an open kernel implies \emph{a fortiori} that a homomorphism is continuous.

Numerous classes of groups have recently been shown to satisfy automatic continuity and the arguments generally follow a diagonalization type argument, as used historically by Specker \cite{Sp}, Higman \cite{H}, Dudley \cite{Du}, Eda \cite{E1} and others.  Free (abelian) groups, Baumslag-Solitar groups, torsion-free word hyperbolic groups, and Thompson's group $F$ are among the currently known groups which are n-, cm-, and lcH-slender (these results are due to various authors, see the introduction to \cite{ConCor} for historical references).  The class of braid groups defined by Artin \cite{A} has heretofore not been known to satisfy automatic continuity conditions \cite[Question 5.3]{ConCor}.  We rectify the situation with the following:

\begin{bigtheorem}\label{braidslender}  
	Every torsion-free virtually poly-free group is n-, cm-, and lcH-slender. In particular, so is each spherical Artin group of type  $A_n$, $B_n$, $D_n$, $I_2(p)$, and $F_4$.
\end{bigtheorem}

The above notions of slenderness are intuitively a second order logical property.  Indeed, for any nontrivial group $G$ there is a countable group $H$ with the same first order theory as $G$ such that $H$ is not n-, cm-, or lcH-slender (see Theorem \ref{notsofast}).  It is therefore a bit surprising that certain first order conditions combined with finite generation allow us to conclude that a group is n-, cm- and lcH-slender.  For example, any torsion-free abelian group which is also finitely generated is isomorphic to a finite rank free abelian group $\mathbb{Z}^n$, and such groups are classically know to be n-, cm-, and lcH-slender \cite{Du}.  It is also fairly easy to see more generally that a torsion-free nilpotent group which is finitely generated is n-, cm-, and lcH-slender \cite{Con}.

The first order theory of nonabelian free groups has been a subject of thorough study over the last several decades and it is therefore natural to ask what might be the situation with the finitely generated groups whose first order theory is that of free groups.  Importantly, free groups are n-, cm- and lcH-slender \cite{Du}.

For any equationally Noetherian group $G$, in particular for the free group, the class of groups which are residually (or fully residually) $G$ play an important role in studying its first order theory: they form the quasi-variety of $G$, that is the set of all groups which satisfy the same quasi-identities as $G$ (satisfy all the universal sentences satisfied by $G$, correspondingly), are coordinate groups of (irreducible) varietes over $G$ etc. We show the following:

\begin{bigtheorem}\label{residualslender}  For a group $G$ the following hold:
\begin{enumerate} 
\item  If $|G|<2^{\aleph_0}$ then $G$ is n-slender if and only if $G$ is residually n-slender.

\item  If $|G|<2^{\aleph_0}$ then $G$ is cm-slender if and only if $G$ is residually cm-slender.

\item  $G$ is lcH-slender if and only if $G$ is residually lcH-slender.
\end{enumerate}
\end{bigtheorem}

Since limit groups are (fully) residually free and free groups are n-, cm-, and lcH-slender it immediately follows that limit groups are as well. Thus, any finitely generated group with the same first order theory or, more generally universal theory of the free group or indeed a group satisfying all the quasi-identities satisfied by a non-abelian free group has all these slenderness conditions. On the other hand, in contrast with the above remarks, we show in Theorem \ref{notsofast} that given n-, cm- or lcH-slender group $G$ there always exists a countable group which is elementarily equivalent to $G$ but is not n-, cm- or lcH-slender. Therefore, Theorems \ref{residualslender} and \ref{notsofast} demonstrate, in particular, that while there is a very strong connection between finitely generated models of the first-order theory of free groups and residually free groups, this connection weakens if one is to study infinitely generated models. 

We note that infinitary logic provides sufficient tools to guarantee slenderness without restricting the number of generators.  For example, the property of being an abelian n-slender (respectively abelian lcH-slender) group is $L_{\infty \omega_1}$ (resp. $L_{\infty \omega}$) definable.  This follows from \cite[Theorem 3.3]{E1}, \cite[Theorem C, (1)]{Cor2} and \cite{ShKo}.  For the nonabelian case, any group which is $L_{\infty \omega_1}$ equivalent to a free group is n-, cm-, and lcH-slender (using \cite[Theorem 0.2]{Me} with \cite[Theorem A]{Cor2}).

We give some background in Section \ref{Background}, prove Theorem \ref{braidslender} in Section \ref{TheoremA}, and prove Theorem \ref{residualslender} in Section \ref{TheoremB}.
\end{section}

\begin{section}{Some background}\label{Background}

We give the definition of the Hawaiian earring group and state some relevant facts.  Let $\{a_n^{\pm 1}\}_{n\in \omega}$ be a countably infinite set where each element has a formal inverse.  A \emph{word} $W$ is a function whose domain is a totally ordered set $\overline{W}$ and whose codomain is $\{a_n^{\pm 1}\}_{n\in \omega}$ such that for each $n\in \omega$ the set $\{i\in \overline{W}\mid W(i) \in \{a_n^{\pm 1}\}\}$ is finite.  Clearly the domain of any word is countable.  We understand two words $W_0$ and $W_1$ to be the same, and write $W_0 \equiv W_1$, if there exists an order isomorphism $\iota: \overline{W_0} \rightarrow \overline{W_1}$ such that $W_1(\iota(i)) = W_0(i)$.  Let $\W$ denote the set of $\equiv$ equivalence classes.  For each $m\in \omega$ we define the function $p_m: \W \rightarrow \W$ by the restriction $p_m(W) = W\upharpoonright\{i\in \overline{W}\mid W(i) \in \{a_{k}^{\pm 1}\}_{k=0}^{m-1}\}$.  Clearly $p_k\circ p_m =p_k$ whenever $k\geq m$.  The word $p_m(W)$ has finite domain, and we write $W_0 \sim W_1$ if for every $m\in \omega$ we have $p_m(W_0)$ equal to $p_m(W_1)$ as elements in the free group over $\{a_{k}\}_{k=0}^{m-1}$.  Write $[W]$ for the equivalence class of $W$ under $\sim$.  

Given two words $W_0$ and $W_1$ we define their concatenation $W_0W_1$ to be the word whose domain is the disjoint union of $\overline{W_0}$ and $\overline{W_1}$ under the order extending that of the two subsets which places elements of $\overline{W_0}$ below those of $\overline{W_1}$, and such that $W_0W_1(i) = \begin{cases}W_0(i)$ if $i\in \overline{W_0}\\W_1(i)$ if $i\in \overline{W_1}\end{cases}$.  Given $W\in \W$ we let $W^{-1}$ be the word whose domain is $\overline{W}$ under the reverse order and such that $W^{-1}(i) = (W(i))^{-1}$.

The set $\W/\sim$ has a group structure defined by letting $[W_0][W_1] = [W_0W_1]$ and $[W]^{-1} = [W^{-1}]$.  The identity element is the $\sim$ class of the empty word $E$.  This group is isomorphic to the fundamental group of the Hawaiian earring and we denote it $\HEG$.  For each $m\in \omega$ the word map $p_m$ defines a retraction homomorphism, also denoted $p_m$, which takes $\HEG$ to a subgroup which is isomorphic to the free group on $\{a_k\}_{k=0}^{m-1}$, which we denote $\HEG_m$.  Again, $p_k\circ p_m = p_k$ whenever $k \geq m$.  The set of all elements of $\HEG$ which have a representative using no letters in $\{a_k^{\pm 1}\}_{k=0}^{m-1}$ is also a retract subgroup, which we denote $\HEG^m$.  There is a natural isomorphism $\HEG \simeq \HEG_m * \HEG^m$.

\begin{definition}\label{nslender}  
	A group $G$ is \emph{noncommutatively slender}, or \emph{n-slender}, if for every homomorphism $\phi: \HEG \rightarrow G$ there exists $m\in \omega$ such that $\phi\circ p_m = \phi$.
\end{definition}

\noindent Equivalently, $G$ is n-slender if for every homomorphism $\phi: \HEG\rightarrow G$ there exists $m\in \omega$ such that $\HEG^m \leq \ker(\phi)$.

\begin{obs}
	The n-slender groups do not contain torsion or $\mathbb{Q}$ as a subgroup (\cite[Theorem 3.3]{E1}, \cite{Sas}).  The same is true for cm- and lcH-slender groups.  The additive group on the real numbers $\mathbb{R}$ has a topology which is both completely metrizable and locally compact Hausdorff and there exists a discontinuous map from $\mathbb{R}$ to $\mathbb{Q}$ constructed by selecting a Hamel basis.  Similarly, a group with torsion includes a cyclic subgroup of prime order $p$ and one can again construct a discontinuous homomorphism from the compact metrizable group $\prod_{\omega}(\mathbb{Z}/p\mathbb{Z})$ to $\mathbb{Z}/p\mathbb{Z}$ by a vector space argument.
\end{obs}

\begin{obs}  
	We observe that $\HEG$ is itself both lcH- and cm-slender.  This is an easy consequence of the following theorems (due to Morris and Nickolas \cite{MoN}, and  Slutsky \cite{Sl}, respectively):

\begin{theorem}\label{MorrisandNickolas}  
	Suppose that $\phi: H \rightarrow *_{i\in I}G_i$ is an abstract homomorphism with $H$ a locally compact Hausdorff group (where $\rightarrow *_{i\in I}G_i$ denotes the free product of the collection of groups $\{G_i\}_{i\in I}$).  Then either $\ker(\phi)$ is open or $\phi(H)$ lies entirely inside a conjugate of one of the $G_i$.
\end{theorem}

\begin{theorem}\label{Slutsky}  
	Suppose that $\phi: H \rightarrow *_{i\in I}G_i$ is an abstract homomorphism with $H$ a completely metrizable group.  Then either $\ker(\phi)$ is open or $\phi(H)$ lies entirely inside a conjugate of one of the $G_i$.
\end{theorem}

\noindent To see that $\HEG$ is lcH-slender we suppose for contradiction that $\phi: H \rightarrow \HEG$ has locally compact Hausdorff domain and $\ker(\phi)$ is not open.  For each $m\in \omega$ we have $\HEG \simeq \HEG_m *\HEG^m$.  Then by Theorem \ref{MorrisandNickolas} we know that $\phi(H)$ is contained in a conjugate of $\HEG_m$ or a conjugate of $\HEG^m$.  Since $\HEG_m$ is lcH-slender (see \cite{Du}) it must be that $\phi(H)$ is actually contained in a conjugate of $\HEG^m$.  Then $p_m \circ \phi: H \rightarrow \HEG$ is trivial for each $m$.  Since for each nontrivial $W\in \HEG$ we have $p_m(W) \neq 1$ for some $m\in \omega$, it follows that $\phi$ is the trivial homomorphism, so $\phi$ had open kernel after all.  Theorem \ref{Slutsky} is used in the cm-slender case.
\end{obs}

\end{section}

\begin{section}{Theorem \ref{braidslender}}\label{TheoremA}

We shall be concerned with solving countable systems of equations in a group.  More particularly, given a group $G$ and sequence $\{a_m\}_{m\in \omega}$ of elements of $G$ and a sequence $\{k_m\}_{m\in \omega}$ of positive natural numbers we consider the system of equations 
$$
y_m = a_my_{m+1}^{k_m}.
$$

\begin{definition} We call a sequence $\{b_m\}_{m\in \omega}$ in $G$ for which $b_m = a_mb_{m+1}^{k_m}$ for all $m\in \omega$ a \emph{solution} to this set of equations.  Similarly for a finite set of equations $\{y_m = a_my_{m+1}^{k_m}\}_{m=0}^{M}$ we call a finite sequence $b_0, \ldots, b_{M+1}$ in $G$ a \emph{solution} provided $b_m = a_mb_{m+1}^{k_m}$ for all $0\leq m\leq M$.
\end{definition}

We shall rely on the following theorem which is rather lengthy to state.

\begin{theorem}\label{Eventualconstantsolutions} The following statements hold:

\begin{enumerate}  \item Suppose $\phi: \HEG \rightarrow G$ is a homomorphism for which there exists $M\in \omega$ such that $M'\geq M$ implies $\phi(\HEG^{M'}) = \phi(\HEG^M)$.  Then given any sequence $\{a_m\}_{m\in \omega}$ in $\phi(\HEG^M)$ and sequence $\{k_m\}_{m\in \omega}$ of positive natural numbers there exists a solution to the system of equations
$$
y_m = a_my_{m+1}^{k_m}
$$

\item Suppose $\phi: H \rightarrow G$ is a homomorphism with $H$ completely metrizable and such that for some neighborhood $V$ of $1_H$ every subneighborhood $Z \subseteq V$ of $1_H$ satisfies $\phi(Z) = \phi(V)$.  Then given any sequence $\{a_m\}_{m\in \omega}$ in $\phi(V)$ and sequence $\{k_m\}_{m\in\omega}$ of positive natural numbers there exists a solution to the defined system of equations $y_m = a_my_{m+1}^{k_m}$ in the group $\phi(V)$.

\item  Suppose $\phi: H \rightarrow G$ is a homomorphism with $H$ locally compact Hausdorff and such that for some neighborhood $V$ of $1_H$ every subneighborhood $Z \subseteq V$ of $1_H$ satisfies $\phi(Z) = \phi(V)$.  Then given any sequence $\{a_m\}_{m\in \omega}$ in $\phi(V)$ and sequence $\{k_m\}_{m\in \omega}$ of positive natural numbers there exists an element $b_0\in \phi(V)$ such that for each $M\in \omega$ there exist $b_1, \ldots, b_{M+1}$ for which $\{b_m\}_{m=0}^{M+1}$ solves the finite system of equations $\{y_m= a_my_{m+1}^{k_m}\}_{m=0}^M$.
\end{enumerate}
\end{theorem}

\begin{proof} (1)  Suppose $\phi: \HEG\rightarrow G$ satisfies the hypotheses and let $\{a_m\}_{m\in \omega}$ be a sequence in $\phi(\HEG^M)$ and $\{k_m\}_{m\in \omega}$ a sequence of positive naturals.  Select $W_0\in \HEG^M \cap \phi^{-1}(a_0)$ and more generally select $W_m\in \HEG^{M+m} \cap \phi^{-1}(a_m)$.  Set $U_m = W_mU_{m+1}^{k_m}$ and notice that the  $U_m$
are in $\HEG^M$ and the sequence $\{\phi(U_m)\}_{m\in \omega}$ is a solution to the set of equations $y_m = a_my_{m+1}^{k_m}$ in $\phi(\HEG^M)$.

(2)  Let $\phi: H \rightarrow G$ be a homomorphism with $H$ completely metrizable, say by the metric $d$, and $V$ be a neighborhood of $1_H$ as in the hypotheses.  Let $\{a_m\}_{m\in\omega} \subseteq \phi(V)$ and let $\{k_m\}_{m\in\omega}$ be a sequence of positive natural numbers and $V_0 = V$.  Pick $h_0\in V\cap \phi^{-1}(a_0)$ and let $\epsilon_0 = d(h_0, H \setminus V_0)$.  Select a neighborhood $V_0 \supseteq V_1$ of $1_H$ such that $h\in V_1$ implies $d(h_0(h)^{k_0}, h_0)\leq \frac{\epsilon_0}{4}$.  Select $h_1\in V_1\cap \phi^{-1}(a_1)$ and let $\epsilon_1 = d(h_1, H\setminus V_1)$.  Supposing we have selected neighborhoods $V_0 \supset \cdots \supset V_m$ of $1_H$, elements $h_0, \ldots, h_m$ and positive real numbers $\epsilon_0, \ldots, \epsilon_m$ in this way we select a neighborhood $V_m \supseteq V_{m+1}$ of $1_H$ such that $h\in V_{m+1}$ implies
$$
\begin{array}{c}
d(h_0(h_1(\dots h_m(h)^{k_m} \dots )^{k_1})^{k_0}, h_0(h_1(\dots h_{m-1}(h_m)^{k_{m-1}}\dots )^{k_1})^{k_0})\leq \frac{\epsilon_0}{2^{m+3}}\\
d(h_1(h_2(\dots h_m(h)^{k_m}\dots )^{k_2})^{k_1}, h_1(h_2(\dots h_{m-1}(h_m)^{k_{m-1}}\dots )^{k_2})^{k_1})\leq \frac{\epsilon_1}{2^{m+3}}\\
\vdots\\
d(h_{m}(h)^{k_m}, h_{m}) \leq \frac{\epsilon_m}{2^{m+3}}
\end{array}
$$
\noindent Select $h_{m+1}\in V_{m+1} \cap \phi^{-1}(a_{m+1})$ and let $\epsilon_{m+1} = d(h_{m+1}, H \setminus V_{m+1})$.  It is straightforward to check that for each $m\in \omega$ the sequence $h_m(h_{m+1}(\cdots h_r^{k_{r-1}}\cdots )^{k_{m+1}})^{k_{m}}$ is Cauchy and converges to an element $y_m$.  It is also clear that $d(y_m, h_m)\leq \frac{\epsilon_m}{2}$ so that in particular $y_m\in V_m$.  Clearly $\{b_m = \phi(y_m)\}_{m\in \omega}$ is a solution to the system of equations.

(3) Let $\phi: H \rightarrow G$ be a homomorphism with $H$ locally compact Hausdorff and $V$ be a neighborhood of $1_H$ as in the hypotheses.  We may assume without loss of generality that $\overline{V}$ is compact.  Let $\{a_m\}_{m\in\omega} \subseteq \phi(V)$ and let $\{k_m\}_{m\in\omega}$ be a sequence of positive natural numbers and select neighborhood $V_0$ of $1_H$ such that $\overline{V_0} \subseteq V$.  Select $h_0\in V_0 \cap \phi^{-1}(a_0)$.  Pick a neighborhood $V_1 \subseteq V_0$ of $1_H$ such that $h\in \overline{V_1}$ implies $h_0h^{k_0}\in V_0$.  Select $h_1\in V_1 \cap \phi^{-1}(a_1)$.  Supposing that we have selected $V_0\supseteq  \ldots \supseteq V_m$ and elements $h_0, \ldots, h_m$ in this way, we select $V_{m+1} \subseteq V_m$ such that $h\in \overline{V_{m+1}}$ implies $h_mh^{k_m}\in V_m$.  Select $h_{m+1}\in V_{m+1} \cap \phi^{-1}(a_{m+1})$.  Letting $K_m = h_0(h_1(\dots h_m(\overline{V_{m+1}})^{k_m}\dots)^{k_1})^{k_0}$ we have $K_m$ as a nesting sequence of nonempty compacta and therefore we may select $y_0 \in \bigcap_{m\in \omega}K_m$.  Let $b_0 = y_0$.  Given an $M\in \omega$ we may select $y_{M+1}\in \overline{V_{M+1}}$ such that $y_0 =  h_0(h_1(\dots h_M(y_{M+1})^{k_M}\dots)^{k_1})^{k_0}$.  For each $1\leq m\leq M$ define $y_m=h_m(\dots h_M(y_{M+1})^{k_M} \dots )^{k_m}$.  Letting $b_m = \phi(y_m)$ for each $1 \leq m\leq M+1$ it is clear that $\{b_m\}_{m=0}^{M+1}$ is a solution to $\{y_m = a_my_{m+1}^{k_m}\}_{m=0}^M$.
\end{proof}

The relevance of Theorem \ref{Eventualconstantsolutions} is given in the following lemma, which is the conjunction of \cite[Theorem 4.12]{CaCon} and \cite[Prop. 3.5]{ConCor}.

\begin{lemma}\label{eventualconstant}  The following hold for a group $G$ satisfying $|G|<2^{\aleph_0}$:

\begin{enumerate}
\item  If $\phi: \HEG \rightarrow G$ is an abstract group homomorphism then there exists some $M\in \omega$ such that $M'\geq M$ implies $\phi(\HEG^{M'}) = \phi(\HEG^M)$.

\item  If $\phi: H \rightarrow G$ is an abstract group homomorphism with $H$ either a completely metrizable or a locally compact Hausdorff topological group then there exists an open neighborhood $U$ of $1_H$ such that for every neighborhood $V\subseteq U$ of $1_H$ we have $\phi(V) = \phi(U)$.
\end{enumerate}

\end{lemma}

\begin{definition}  We say that a group $G$ satisfies conditions (*) if 

\begin{enumerate}
\item $|G|<2^{\aleph_0}$;

\item $G$ is torsion-free;

\item There exists a $j\in \omega \setminus \{0\}$ for which every $g\in G\setminus \{1_G\}$ has a sequence $\{p_m\}_{m\in \omega}$ such that for every sequence of nonzero naturals $\{k_m\}_{m\in \omega}$ for which $p_m$ divides $k_m$ the system of equations $y_m = g^jy_{m+1}^{k_m}$ has no solution in $G$.

\end{enumerate}

\end{definition}

\begin{definition}  We say $G$ satisfies conditions (**) if 

\begin{enumerate}
\item $|G|<2^{\aleph_0}$;

\item $G$ is torsion-free; 

\item There exists $j\in \omega\setminus \{0\}$ for which every $g\in G\setminus \{1_G\}$ has a sequence $\{p_m\}_{m\in \omega}$ such that for every sequence of nonzero naturals $\{k_m\}_{m\in \omega}$ for which $p_m$ divides $k_n$ and for every $b_0\in G$ there exists some $M\in \omega$ for which there do not exist $b_1, \ldots, b_{M+1}$ for which $\{b_m\}_{m=0}^{M+1}$ is a solution to the finite system of equations $\{y_m = g^jy_{m+1}\}_{m=0}^M$.

\end{enumerate}
\end{definition}

It is clear that conditions (**) imply conditions (*).  We will use conditions (**) for our proof of Theorem \ref{braidslender}.  We include mention of both conditions, however, because it is still unknown whether n-, cm- and lcH-slenderness are interchangeable for groups of cardinality $<2^{\aleph_0}$.  Among abelian groups of cardinality $<2^{\aleph_0}$ these conditions are all equivalent (see \cite[Theorem C]{ConCor}).  Also, there is known to exist a countable group which is n-, cm-, and lcH-slender and which does not satisfy (**), see Remark \ref{Whyboth} below.  It is unclear whether there is a countable group which is slender in one of these senses and which fails to satisfy (*).

\begin{lemma}\label{nicestar}  Any group satisfying (*) is n- and cm-slender.  Any group satisfying (**) is lcH-slender.
\end{lemma}

\begin{proof}  Suppose $G$ satisfies (*) and that $\phi: \HEG \rightarrow G$ is a homomorphism.  Since $|G|<2^{\aleph_0}$ we have by Lemma \ref{eventualconstant} the existence of an $M\in \omega$ for which $M'\geq M$ implies $\phi(\HEG^{M'}) = \phi(\HEG^M)$.  If $\phi(\HEG^M)$ is nontrivial then since $G$ satisfies (*) we take $j$ and $g\in \phi(\HEG^M)\setminus \{1_G\}$ and $\{p_m\}_{m\in \omega}$ as given by (*).  Then the system of equations $y_m = g^jy_{m+1}^{p_m}$ has no solution in $G$, and therefore no solution in $\phi(\HEG^M)$, contrary to Theorem \ref{Eventualconstantsolutions}.

The proof where the domain is changed to be completely metrizable is entirely analogous, and the proof in the (**) case follows by appropriately modifying the proof and using Lemma \ref{eventualconstant} and Theorem \ref{Eventualconstantsolutions}.
\end{proof}

\begin{lemma}\label{norootsfordoublestar}  If $G$ satisfies (**) then there cannot exist $g\in G\setminus \{1_G\}$ such that for each $n\in \omega \setminus \{0\}$ there exists $h_n\in G$ such that $h_n^n = g$.  
\end{lemma}

\begin{proof}  Suppose on the contrary that such a $g$ exists.  Pick $j$ and sequence $\{p_m\}_{m\in \omega}$ since $G$ satisfies (**).  For each $m\in \omega$ let $k_m \geq 1$ with $p_m \mid k_m$.  Let $b_0 = g$.  For a fixed $M\in \omega$ pick $h\in G$ such that $h^{k_0\cdots k_M} = g$.  For each $0 \leq m\leq M+1$ we let $b_m = h^{k_M\cdots k_{m}(1-j-jk_1-jk_2- \ldots -jk_{m-1})}$.  It is straightforward to check that $b_0, \ldots, b_{M+1}$ is a solution to $\{y_m = g^jy_{m+1}^{k_m}\}_{m=0}^M$.
\end{proof}

We provide a couple of lemmas demonstrating some desirable closure properties.

\begin{lemma}\label{extensions}  
	If $1 \rightarrow N \rightarrow G \rightarrow^q Q \rightarrow 1$ is a short exact sequence and both $N$ and $Q$ satisfy (*) then $G$ also satisfies (*).  The same holds when (*) is replaced with (**).
\end{lemma}

\begin{proof}  
	Assume the hypotheses.  Certainly such a $G$ must be torsion-free and of appropriate cardinality.  Let $j_N$ and $j_Q$ be the constants for $N$ and $Q$ respectively in the definition of (*).  We let $j = j_Qj_N$.  Suppose $g\in G \setminus \{1_G\}$.  In case $q(g) \neq 1_Q$ we let $\{p_m\}_{m\in \omega}$ be the sequence determined for the nontrivial element $q(g^j)$ in $Q$.  Suppose $p_m$ divides $k_m\neq 0$.  Any solution to the equations $y_m = g^jy_{m+1}^{k_m}$ projects under $q$ to a solution of $y_m = q(g^j)y_{m+1}^{k_m}$ in $Q$, which does not exist.

Suppose now that $q(g) = 1_Q$.  Then $g\in N$.  Under the definition of (*) select a sequence $\{p_m'\}_{m\in \omega}$ for $g^j$ as an element of $N$.  Let $p_m = (m+2)p_m'$ and let $\{k_m\}_{m\in\omega}$ be a sequence of positive naturals for which $p_m$ divides $k_m$.  Suppose that $y_m = g^jy_{m+1}^{k_m}$ has a solution $\{b_m\}_{m\in \omega}$ in $G$.  Since $q(g) = 1_Q$ we get that $q(b_m) = q(b_{m+1})^{k_m}$.  Since $Q$ is torsion-free we know $q(b_0) \neq 1_Q$ if and only if $q(b_m) \neq 1_Q$ for all $m\in \omega$.  For contradiction suppose that $q(b_0) \neq 1_Q$.  Since $p_m$ divides each $k_m$ we get
$$
\begin{array}{c}
q(b_0) = (q(b_1)^{\frac{k_0}{2}})^2\\
q(b_1)^{\frac{k_0}{2}} = (q(b_2)^{\frac{k_0}{2}\frac{k_1}{3}})^3\\
q(b_2)^{\frac{k_0}{2}\frac{k_1}{3}} = (q(b_3)^{\frac{k_0}{2}\frac{k_1}{3}\frac{k_2}{4}})^4\\
\vdots
\end{array}
$$
\noindent  Then we have an isomorphic copy of the group $\mathbb{Q}$ inside of $Q$.  But $Q$ is n-slender by Lemma \ref{nicestar} and we have a contradiction.  

Thus all the $b_m$ lie inside $N$ and are a solution to $y_m = g^jy_{m+1}^{k_m}$, a contradiction.

Now suppose the appropriate hypotheses for the claim regarding (**).  Let $j = j_Qj_N$.  Suppose $g\in G\setminus \{1_G\}$.  If $q(g) \neq 1_Q$ select $\{p_m\}_{m\in \omega}$ for $q(g^j)$ as in the definition of (**) for the group $Q$.  Given $b_0\in G$ we select $M$ such that such that the system of equations $\{y_m = q(g^j)y_{m+1}^{k_m}\}_{m=0}^M$ does not have a solution in $Q$ with $y_0 = q(b_0)$.  Then the system of equations $\{y_m = gy_{m+1}^{k_m}\}_{m=0}^M$ cannot have a solution in $G$.

Suppose now that $g\in N$.  Let $\{p_m'\}_{m\in \omega}$ be a sequence in $\omega \setminus \{0\}$ satisfying the conditions of (**) in $N$.  Let $p_m = (m+1)p_m'$.  Let $\{k_m\}_{m\in \omega}$ be a sequence of nonzero naturals such that $p_m\mid k_m$.  Suppose that for some $b_0\in G$ the system of equations $\{y_m = g^jy_{m+1}^{k_m}\}_{m=0}^M$ always has a solution in $G$ with $b_0 = y_0$ for each $M\in \omega$.  If $b_0 \in N$ then any solution of $\{y_m = g^jy_{m+1}^{k_m}\}_{m=0}^{M}$ has all $b_1, \ldots, b_{M+1}$ in $N$ as well since $Q$ is torsion-free.  This cannot be so, and thus $b_0 \notin \ker(q)$.

Since $q(b_0)\neq 1_Q$ and $q(g) = 1_Q$ we see that $\{y_m = y_{m+1}^{k_m}\}_{m=0}^M$ always has a solution in $Q$ with $y_0 = q(b_0)$.  Then arguing as in the (*) case, we see that $q(b_0)$ has an $n$-th root in $Q$ for each $n$.  This contradicts Lemma \ref{norootsfordoublestar}.
\end{proof}

For the next lemma we recall that a group $G$ is of \emph{bounded exponent} if there exists some $d\geq 1$ such that $g^d = 1_G$ for all $g\in G$.  Obviously a group of bounded exponent is torsion, and any finite group is of bounded exponent. 

\begin{lemma}\label{finiteextension}  Suppose $1 \rightarrow N \rightarrow G \rightarrow Q \rightarrow 1$ is a short exact sequence, $N$ satisfies (*), $G$ is torsion-free, and $Q$ of bounded exponent and $|Q|<2^{\aleph_0}$.  Then $G$ satisfies (*).  Similarly when (**) replaces (*).
\end{lemma}

\begin{proof}  Such a $G$ is certainly of appropriate cardinality, and torsion-free by assumption.  Let $j_N$ be an exponent as in the definition of (*) for $N$.  Let $d \geq 1$ be such that $h^d = 1_Q$ for all $h\in Q$.  Let $j = j_Nd$.  Let $g\in G\setminus \{1_G\}$ be given.  Pick a sequence $\{p_m'\}_{m\in \omega}$ for $g^{j}$ as an element of $N$.  Let $p_m = dp_m'$.  Suppose a nonzero sequence $\{k_m\}_{m\in \omega}$ satisfies $p_m|k_m$.

Consider the system $y_m = g^j y_{m+1}^{k_m}$.  Any solution $\{b_m\}_{m\in \omega}$ in $G$ is in fact a solution \emph{in $N$} to the system $y_m = g^jy_{m+1}^{k_m}$ (that each $b_m\in N$ follows from the fact that $b_m$ is a product of two elements such that each is a $d$ power in $G$).  This cannot exist by how $\{p_m'\}_{m\in \omega}$ was chosen.

Now suppose the appropriate hypotheses for proving the claim regarding (**).  Again, $G$ is of correct cardinality and torsion-free.  We let $j = j_Nd$.  Let $g\in G$ and let $\{p_m'\}_{m\in \omega}$ work for $g^j$ in $N$.  Let $p_m = p_m'd$ and let $\{k_m\}_{m\in \omega}$ be a nonzero sequence for which $p_m\mid k_m$.  Let $b_0\in G$ be given.  Suppose that for each $M\in \omega$ the system $\{y_m = g^jy_{m+1}^{k_m}\}_{m=0}^M$ has a solution $b_0, \ldots, b_{M+1}$ in $G$ such that $b_0 = y_0$.  Then $b_0, \ldots, b_N$ are all in $N$.  Then for each $M$ the system $\{y_m = g^jy_{m+1}^{k_m}\}_{m=0}^M$ always has a solution in $N$, contradicting the choice of $\{p_m'\}_{m\in \omega}$.
\end{proof}

Recall that a \textit{length function} on a group $G$ is a function $L: G \rightarrow [0, \infty)$ such that for all $g, h \in G$ we have

\begin{enumerate}
\item $L(1_G) = 0$

\item $L(g) = L(g^{-1})$

\item $L(gh) \leq L(g) + L(h)$
\end{enumerate}

\noindent A length function $L$ is a \textit{Dudley norm} if it takes only natural number values and for each $g\neq 1_G$ and $n\in \omega$ we have $L(g^n) \geq \max\{n, L(g)\}$ (see \cite{Du}).

\begin{lemma}  If $|G|<2^{\aleph_0}$ and $G$ has a Dudley norm then $G$ satisfies (**).

\end{lemma}

\begin{proof}  It is clear that any such $G$ must be torsion-free.  Thus the only condition in question within (**) is part (3).  First let $L$ be a Dudley norm.  Let $j = 1$ and $g\in G\setminus \{1_G\}$ be given.  Let $p_m = (m+2)(L(g)+1)$.  Let $k_m$ be a sequence of positive naturals such that $p_m$ divides $k_m$.  In particular $k_m \geq p_m$.  Suppose $b_0 \in G$ and let $M = L(b_0) +1$.  Suppose for contradiction that $b_0, \ldots, b_{M+1}$ is a solution to $\{y_m = gy_{m+1}\}_{m=0}^M$.  If $b_{M+1} = 1_G$ we have
$$
L(b_M^{k_{M-1}}) = L((gb_{M+1}^{k_M})^{k_{M-1}}) = L(g^{k_{M-1}})= k_{M-1} \geq (M+1)(L(g) +1)
$$
since $k_{M-1} \geq p_{M-1} = (M+1)(L(g) +1)\geq L(g)$.  If $b_{M+1}\neq 1_G$ then 
\begin{gather}\notag
\begin{split}
L(b_M^{k_{M-1}}) = & L((gb_{M+1}^{k_M})^{k_{M-1}}) \geq L(gb_{M+1}^{k_M})\geq L(b_{M+1}^{k_M}) - L(g) \geq k_M - L(g) \geq \\
\geq &  p_M -L(g)\geq (M+2)(L(g) +1) - L(g) \geq (M+1)(L(g) +1).
\end{split}
\end{gather}

Thus in either case we have $L(b_M^{k_{M-1}}) \geq (M+1)(L(g) +1)$, and so we can write
\begin{gather}\notag
\begin{split}
L(b_0) =& L(gb_1^{k_0}) \geq L(b_1^{k_0})-L(g)\\
=& L((gb_2^{k_1})^{k_0}) - L(g)\\
\geq& L(b_2^{k_1}) - 2L(g)\\
\geq& \quad \dots\\
\geq& L(b_{M-1}^{k_{M-2}})-(M-1)L(g)\\
\geq& L(b_{M-1}) - (M-1)L(g)\\
=& L(gb_{M}^{k_{M-1}}) - (M-1)L(g)\\
\geq& L(b_M^{k_{M-1}})-ML(g)\\
\geq& (M+1)(L(g) + 1)-ML(g) = L(g)+ 1 + M = L(g) + 2 +L(b_0)\\
\end{split}
\end{gather}
\noindent which is a contradiction.
\end{proof}

We recall the definition of a graph product of groups. Suppose $\Gamma = (V, E)$ is a graph (we allow the sets of vertices and edges to be of arbitrary cardinality but do not allow multiple edges or loops).  To each vertex $v\in V$ we associate a group $G_v$.  The groups $G_v$ are the \emph{vertex groups}.  The graph product $G = \Gamma(\{G_v\}_{v\in V})$ is defined by taking the free product $\ast_{v\in V} G_v$ and taking the quotient by the normal closure of the set $\{[g_{v_0}, g_{v_1}]\}_{g_{v_0} \in G_{v_0}, g_{v_1} \in G_{v_1}, \{v_0, v_1\}\in E}$.  Thus free products of groups and direct sums of groups are examples of graph products of groups, with the graphs having either no edges or being complete in the respective cases.

Each $G_v$ is a retract subgroup of $G$ and $G$ is generated by the elements of the vertex subgroups $G_v$.  Thus each element $g\in G$ has a representation as a word $g =_{G} g_0 g_1g_2\cdots g_{n-1}$, where each $g_i$ is an element of a vertex group.  In such a word we call each $g_i$ a \emph{syllable}.  Given two vertex groups $G_{v_0}$ and $G_{v_1}$ it is easy to see that the subgroup $\langle G_{v_0}\cup G_{v_1}\rangle \leq G$ is a retract of $G$ and is either isomorphic to $G_{v_0} \ast G_{v_1}$ or $G_{v_0}\times G_{v_1}$, the first being the case if and only if $\{v_0, v_1\} \notin E$.  Thus for nontrivial elements $g_0\in G_{v_0}$ and $g_1\in G_{v_1}$ we have that $[g_0, g_1] = 1$ if and only if $\{v_0, v_1\} \in E$.

Following \cite{Gre}, we say that a word $g_0 g_1\cdots g_{n-1}$ is reduced if the following hold:
\begin{enumerate}
\item Each $g_i$ is a nontrivial element of a vertex group and $g_i$ and $g_{i+1}$ are in different vertex groups for all $0\leq i<n-1$
\item If $i \leq k < j$ and 
$$
[g_i, g_{i+1}] = [g_i, g_{i+2}]= \dots = [g_i, g_k] = 1 = [g_{k+1}, g_j] = [g_{k+2}, g_j] = \dots = [g_{j-1}, g_j],
$$ 
then $g_i$ and $g_j$ are in different vertex groups.
\end{enumerate}
  
We say that two reduced words $w_0, w_1$ are equivalent,  if one can obtain $w_1$ from $w_0$ by a permutation of syllables as allowed in the group (that is one can permute the syllables $g_i$ and $g_{i+1}$ if and only if $[g_i, g_{i+1}] =1$).  Clearly the equivalence of $w_0$ to $w_1$ implies that $w_0$ and $w_1$ have the same word length and $w_0 =_G w_1$.  Using $\bigcup_{v\in V} G_v$ as a generating set for $G$ we get a length function $l$ on $G$.

The following result combines the statements of Theorem 3.9 and Corollary 3.13 of \cite{Gre}:
\begin{lemma}\label{Green}  
	Each $g \neq 1$ has a reduced word representation $g =_G g_0g_1\cdots g_{n-1}$ which is unique up to equivalence, with $l(g) = n$.
\end{lemma}

The following lemma strengthens \cite[Lemma 2]{Du}.

\begin{lemma}\label{graphDud}  If each of the groups $\{G_v\}_{v\in V}$ has a Dudley norm then the graph product $G = \Gamma(\{G_v\}_{v\in V})$ also has a Dudley norm.
\end{lemma}
\begin{proof}  
For each $v\in V$ let $L_v$ be a Dudley norm for $G_v$.  We define a length function $L: G\rightarrow \omega$.  Given $g =_G g_0g_1\cdots g_{n-1}$ in reduced form we let $L(g) = \sum_{i=0}^{n-1} L_{v_i}(g_i)$ where $g_i\in G_{v_i}$.  Certainly $L(1_G) = 0$ and $L(g) = L(g^{-1})$. The triangle inequality follows immediately from  \cite[Theorem 2.8.7.]{Goda} and definition of $L$ on $G$. This completes the proof of the triangle inequality and we have that $L$ is indeed a length function.  

It remains to show that $L(g^n) \geq \max\{n, L(g)\}$ for each $g\neq 1$ and $n\in \omega$.  We will use the fact along the way that each of the groups $G_v$ is torsion-free.

Let $g\neq 1$ and write $g =_G g_0\cdots g_{m-1}$ in reduced form. By \cite[Lemma 2.11.1]{Goda}, upto relabling, $g$ can be uniquely written as 
$$
g=(g_0\cdots g_p) \cdot (g_{p+1}\cdots g_{m-(p+2)})\cdot (g_{m-(p+1)}\cdots g_{m-1}),
$$
where $g_i=g_{m-(i+1)}^{-1}$ for all $i=0,\dots, p$ and $g_{p+1}\cdots g_{m-(p+2)}$ is cyclically reduced (see \cite{Gre, Goda} for definition). By definition, it follows that
$$
g^n=g_0\cdots g_p \cdot (g_{p+1}\cdots g_{m-(p+2)})^n\cdot g_{m-(p+1)}\cdots g_{m-1}
$$
is a reduced form for $g^n$. 
If all the syllables of $g_{p+1}\cdots g_{m-(p+2)}$ pairwise commute, then 
$$
L(g^n) = 2L(g_0\cdots g_p) + L(g_{p+1}^n\cdots g_{m-(p+2)}^n)= 2L(g_0\cdots g_p) + \sum_{i = p+1}^{m-(p+2)} L_{v_i}(g_i^n).
$$
The last expression is at least $n$ since $L_{v_{p+1}}(g_{p+1}^n) \geq n$.  Since $L_{v_i}(g_i^n) \geq L_{v_i}(g_i)$ for each $p+1\leq i\leq m-(p+2)$ this last expression is at least $L(g)$ as well.  Thus the desired inequality holds in this case.

Assume now that the syllables of $g_{p+1}\cdots g_{m-(p+2)}$ do not pairwise commute. In this case, the reduced form of $(g_{p+1}\cdots g_{m-(p+2)})^n$ has no less than $2n$ syllables. Hence, $g^n$ has $L$ length $\geq n$ and that $L(g^n)\geq L(g)$.  The proof is now complete.
\end{proof}

\begin{corollary}\label{RAAGS}  
	Right-angled Artin groups of cardinality $<2^{\aleph_0}$ have a Dudley norm and thus satisfy (**).
\end{corollary}

We also have the following (compare \cite[Theorem 3.3]{ConCor}):

\begin{lemma}\label{graphprodstar}  
	If $\Gamma$ has finitely many vertices and each $G_v$ satisfies (*) then the graph product $G = \Gamma(\{G_v\}_{v\in V})$ also satisfies (*).  The similar claim holds for (**).
\end{lemma}
\begin{proof}  We prove the (*) claim, and the (**) claim follows along precisely the same lines.  Assume the hypotheses.  Let $\sigma: G \rightarrow \bigoplus_{v\in V} G_v$ be the map which adds relators causing all elements in distinct vertex groups to commute.  We have by Lemma \ref{extensions} and by induction on $|V|$ that the group $\bigoplus_{v\in V} G_v$ satisfies (*).  Then by Lemma \ref{extensions} it suffices to prove that $\ker(\sigma)$ satisfies (*).  We will show, in fact, that the length function $l$ which counts the number of syllables is a Dudley norm when restricted to $\ker(\sigma)$.  We must show that for $n\in \omega$ and $g\in \ker(\sigma) \setminus \{1\}$ we have $l(g) \geq \max\{n, l(g)\}$.

Suppose $g\in\ker(\sigma) \setminus \{1\}$ has reduced form $g =_G g_0\cdots g_{m-1}$.  As in the proof of Lemma \ref{graphDud}, upto relabling, $g$ can be uniquely written as 
$$
g=(g_0\cdots g_p) \cdot (g_{p+1}\cdots g_{m-(p+2)})\cdot (g_{m-(p+1)}\cdots g_{m-1}),
$$
where $g_i=g_{m-(i+1)}^{-1}$ for all $i=0,\dots, p$ and $g_{p+1}\cdots g_{m-(p+2)}$ is cyclically reduced. By definition, it follows that
$$
g^n=g_0\cdots g_p \cdot (g_{p+1}\cdots g_{m-(p+2)})^n\cdot g_{m-(p+1)}\cdots g_{m-1}
$$
is a reduced form for $g^n$. 
Since $g\in\ker(\sigma)$, so the syllables of $g_{p+1}\cdots g_{m-(p+2)}$ do not commute pairwise. As we already noticed above, in this case the reduced form of $(g_{p+1}\cdots g_{m-(p+2)})^n$ has no less than $2n$ syllables. Thus the length is a Dudley norm.
\end{proof}

We give the following definition for establishing still another class of groups which satisfy (**):

\begin{definition}  For each $g\in G$ we let 
$$
\Roots(g) = \{h\in G\mid (\exists n\in \omega\setminus\{0\})[h^n = g]\}.
$$
We say that a group $G$ \textit{has finite roots} if for each $g\in G$ the set $\Roots(g)$ is finite.
\end{definition}

\begin{lemma}\label{finiteroots}  If $G$ is countable torsion-free group having finite roots then $G$ satisfies (**).
\end{lemma}
\begin{proof}  Assume $G$ satisfies the hypotheses.  Given any set $X\subseteq G$ we let $\Roots(X) = \bigcup_{g\in G}\Roots(g)$.  Notice that for any finite $X\subseteq G$ there exists $p\in \omega \setminus \{0\}$ such that for all $k\geq p$ and $h\neq 1_G$ we have $h^k \notin X$.

Let $G = \{g_0, \ldots\}$ be an enumeration and for each $m\in \omega$ let $X_m = \{g_0, \ldots, g_m\}$.  Let $g\neq 1_G$ be given.  Select $p_0$ such that for all $k_0 \geq p_0$ and $h\neq 1_G$ we have $h^{k_0}\notin g^{-1}X_0$.  Select $p_1$ such that for all $k_1 \geq p_1$ and $h\neq 1_G$ we have $h^{k_1} \notin g^{-1}\Roots(g^{-1}X_1)$ and generally for $m\in \omega$ select $p_m$ such that for $k_m \geq p_m$ we have $h \neq 1_G$ implies $h^{k_m}\notin g^{-1}\Roots(g^{-1}\Roots(\dots g^{-1}X_m \dots))$, where the operator $\Roots$ appears $m$ times in the expression.

Let $b_0\in G$ be given and suppose $b_0 = g_l$.  Let $M = l+1$ and suppose for contradiction that $b_0, \ldots, b_{M+1}$ is a solution to the set of equations $\{y_m = gy_{m+1}^{k_m}\}_{m=0}^M$.  Since $b_m = gb_{m+1}^{k_m}$ for $0 \leq m\leq M$ we have $b_{m+1} \in \Roots(g^{-1}b_m)$, and so for each $1 \leq m \leq M+1$ we have
$$
b_m \in \Roots(g^{-1}\Roots(g^{-1}\dots  \Roots(g^{-1}b_0)\dots))
$$
\noindent where the operator $\Roots$ appears $m$ times.  In particular
$$
b_l \in \Roots(g^{-1}\Roots (g^{-1} \dots  \Roots(g^{-1}X_l)\dots   ))
$$
\noindent where the operator $\Roots$ appears $l$ times, and so 
$$
b_{l+1}^{k_l} \in g^{-1}\Roots(g^{-1}\Roots(\dots \Roots(g^{-1}X_l) \dots))
$$
\noindent Then $b_{l+1} = 1_G$ by the choice of $p_l$.  We also have
$$
b_{l+1} \in \Roots(g^{-1}\Roots(g^{-1}\dots \Roots(g^{-1}X_{l+1})    \dots)),
$$
\noindent where the operator $\Roots$ appears $l+1$ times, so that
$$
b_{l+2}^{k_{l+1}} \in g^{-1}\Roots(\dots \Roots(g^{-1}X_{l+1}) \dots)
$$
\noindent and again $b_{l+2} = 1_G$.  Now $1_G = b_{l+1} = gb_{l+2}^{k_{l+1}} = g$, a contradiction.
\end{proof}

\begin{corollary}  \
\begin{enumerate}
	\item Let $G$ be a torsion-free group hyperbolic relative to a finite set of subgroups $\{H_\lambda\}$. Suppose that each $H_\lambda$ has finite root extraction, then $G$ satisfies (**).
	\item Thompson's group $F$ satisfies (**).
\end{enumerate}
\end{corollary}
\begin{proof}  
Relatively hyperbolic group $G$ has finite root extraction for hyperbolic elements (see  \cite[Theorem 1.14]{O}), and so if every $H_\lambda$ has finite root extraction, then so does $G$ and hence it satisfies (**) by Lemma \ref{finiteroots}.  

Thompson's group $F$ is also torsion-free and has finite roots (see \cite{GSap} Theorem 15.11 and Lemma 15.29) and is also countable, so we again apply Lemma \ref{finiteroots}.
\end{proof}

We may now apply the results of this section to prove

\begin{A}\label{pureandstandardbraid}  
	All torsion-free virtually poly-free groups satisfy (**). In particular, so do spherical Artin groups of type  $A_n$, $B_n$, $D_n$, $I_2(p)$, and $F_4$.
\end{A}
\begin{proof}
Indeed, poly-free groups satisfy (**) by Corollary \ref{RAAGS} and Lemma \ref{extensions}. Torsion-free virtually poly-free groups satisfy (**) by Lemma \ref{finiteextension}. Finally, spherical Artin groups of type $A_n$, $B_n$, $D_n$, $I_2(p)$ and $F_4$ are virtually poly-free, see \cite{brieskorn}. 
\end{proof}

\begin{remark}\label{Whyboth}  
Interestingly there exist countable groups which are n-, cm-, and lcH-slender which do not satisfy (**).  A countable group $G$ is constructed in \cite[Example 3]{Cor1} which is an amalgamated free product of copies of $\mathbb{Z}$ which has a nontrivial element having an $n$-th root for each $n\in \omega \setminus\{0\}$.  We relate the construction of this group and demonstrate why it satisfies the aforementioned conditions.

For each $n\geq 1$ let $G_n$ be a copy of the integer group $\mathbb{Z}$.  For each $n\geq 2$ we let $H_n$ be the unique subgroup of $G_n$ of index $n$.  For each $n\geq 2$ let $\psi_n: H_n \rightarrow G_1$ be an isomorphism (there are only two, pick one).  Let $G$ be the amalgamated free product over the groups $\{G_n\}_{n\geq 1}$ obtained by identifying all $H_n$ to $G_1$ via the maps $\psi_n$.  The subgroup $G_1$ is central in $G$.  Also, the generating element $1 \in G_1$ clearly has an $n$-th root in $G$ for each $n\in \omega \setminus \{0\}$, so by Lemma \ref{norootsfordoublestar} the group $G$ does not satisfy (**).  Notice however that $G$ fits in the short exact sequence
$$
1 \rightarrow G_1 \rightarrow G \rightarrow^q *_{n\geq 2}(\mathbb{Z}/n\mathbb{Z}) \rightarrow 1.
$$
Now suppose $\phi: H \rightarrow G$ is a homomorphism with $H$ completely metrizable.  If $q\circ \phi(H)$ lies entirely in a conjugate of, say, $\mathbb{Z}/n\mathbb{Z}$ then $\phi(H)$ lies entirely in a conjugate of $G_n$.  Since $G_n \simeq \mathbb{Z}$ we immediately see that $\ker(\phi)$ is open since $\mathbb{Z}$ is cm-slender.  Otherwise we have by Theorem \ref{Slutsky} that $\ker(q\circ \phi)$ is open in $H$.  Now $\ker(q\circ\phi)$ is a completely metrizable group, $\ker(\phi) \leq \ker(q \circ \phi)$, and $\phi\upharpoonright \ker(q\circ \phi)$ has image in $\ker(q) = G_1$.  But since $G_1 \simeq \mathbb{Z}$ we get that $\ker(\phi)$ is open in $\ker(q\circ \phi)$, and therefore also open in $H$.  Either way we conclude that $\ker(\phi)$ is open in $H$.  The proof when $H$ is locally compact Hausdorff is entirely analogous.  When $H$ is $\HEG$ we use \cite[Theorem 1.3]{E2} (as is done in \cite[Theorem 4.4]{Cor1}).
\end{remark}

The group $B_m$ embeds as a subgroup of $B_{m+1}$ naturally by fixing a strand in $B_{m+1}$ and having $B_m$ warp all other strands appropriately.  It is interesting to note that our proof does not lend itself to a proof that the direct limit $\varinjlim B_m$ is slender.

\begin{question}  Does n-, cm-, or lcH-slenderness hold for the group $\varinjlim B_m$?
\end{question}

\end{section}

\begin{section}{Theorem \ref{residualslender}}\label{TheoremB}

We state a result which follows immediately from a structure theorem of Cleary and Morris (see \cite[Theorem 1, Remark (ii)]{ClMo}).

\begin{theorem}\label{Morris}  
	If $G$ is a locally compact Hausdorff group then $G$ is homeomorphic to a product $\mathbb{R}^n \times K \times D$ where $K$ is a compact subgroup of $G$ and $D$ is a discrete space.
\end{theorem}

\begin{lemma}\label{closedkernel}  
	If $\phi: H \rightarrow G$ is a group homomorphism with $H$ a completely metrizable group (respectively locally compact Hausdorff group) and $G$ is residually cm-slender (resp. residually lcH-slender) then the kernel of $\phi$ is closed.
\end{lemma}
\begin{proof}  
	In either case, the kernel of $\phi$ is the intersection of all $\ker(\psi \circ \phi)$ where $\psi$ is a homomorphism from $G$ to a slender group.  Each $\ker(\psi \circ \phi)$ is a clopen subgroup of $H$ and the result follows.
\end{proof}

\begin{proof}[Proof of Theorem \ref{residualslender}]  We note that part (1) was proved in \cite{Cor1} using different techniques, but we give a proof here as well for the sake of completeness.  Certainly any notion of slenderness implies the residual notion of slenderness.  Supposing $|G|<2^{\aleph_0}$ and $G$ is residually n-slender we suppose for contradiction that $G$ is not n-slender.  Then we have a map $\phi: \HEG \rightarrow G$ for which $\phi(\HEG^M) \neq\{1_G\}$ for all $M$.  By Lemma \ref{eventualconstant} (1) we select an $M\in \omega$ for which $M'\geq M$ implies $\phi(\HEG^{M'}) = \phi(\HEG^M)$.  Since $G$ is residually n-slender and $\phi(\HEG^M) \neq \{1_G\}$ there exists a nontrivial map $\psi: G \rightarrow L$ with $L$ an n-slender group and $\psi\upharpoonright \phi(\HEG^M)$ a nontrivial homomorphism.  But now $\psi \circ \phi$ is a nontrivial homomorphism from $\HEG$ to $L$ for which $\psi\circ \phi(\HEG^{M'})$ is never trivial for any $M'\in \omega$, a contradiction.  Thus we have finished the proof of (1), and the proof of (2) is entirely analogous.  It remains to prove the nontrivial direction of part (3).

Suppose $G$ is residually lcH-slender and let $\phi:H \rightarrow G$ be an abstract group homomorphism with $H$ a locally compact Hausdorff group.  The kernel $\ker(\phi)$ is closed by Lemma \ref{closedkernel}, and so $H/\ker(\phi)$ is easily locally compact Hausdorff and the map $\phi$ descends to an injection from $H/\ker(\phi)$ to $G$.  We shall show that $H/\ker(\phi)$ is discrete and we will be done.  Thus it suffices to show that if such a map $\phi: H\rightarrow G$ is injective then the domain is discrete.  

We notice that if $K \leq H$ is any compact subgroup then $\phi(K)$ must be trivial.  Indeed if $g\in \phi(K) \setminus \{1_G\}$ we take a homomorphism $\psi:G \rightarrow L$ with $L$ an lcH-slender group and $g \notin \ker(\psi)$.  Now $\psi \circ \phi\upharpoonright K$ has open kernel and the image is infinite (since it is a nontrivial subgroup of $L$ and $L$ is torsion-free).  But this gives an infinite open covering of $K$ by the cosets of $\ker(\psi \circ \phi\upharpoonright K)$, and this cover can certainly not have a finite subcover.  Since we are now assuming $\phi$ is injective, we now know that all compact subgroups of $H$ are trivial.

Now by Theorem \ref{Morris} we have $H$ is homeomorphic to a product $\mathbb{R}^n \times D$ with $D$ discrete.  Fix an element $y\in D$.  Since any homomorphism $\beth$ from $H$ to an lcH-slender group is continuous with clopen kernel, we know $\beth(\mathbb{R}^n \times \{y\})$ always maps to a single point.  Since $G$ is residually lcH-slender this means that $n = 0$, for otherwise we could not separate distinct elements of $\mathbb{R}^n \times \{y\}$ using a composition $\psi\circ \phi$ with $\psi: G \rightarrow L$ having lcH-slender codomain, and $\phi$ is assumed to be injective.  Thus $H$ is discrete and we are done with the proof of Theorem \ref{residualslender}.
\end{proof}

Notice that parts (1) and (2) of this theorem cannot be improved since the group $\prod_{\omega}\mathbb{Z}$ is residually n-, cm-, and lcH-slender but is neither n- nor cm-slender.  We now show that the finite generation of limit groups is essential in concluding n-, cm-, and lcH-slenderness.

\begin{theorem}\label{notsofast}  
	If $G$ is a nontrivial group there exists a countable group $H$ having the same first order theory as $G$ such that $H$ is not n-, cm-, or lcH-slender.
\end{theorem}
\begin{proof}  Suppose first that $G$ has torsion.  Let $H$ be any countable group with the same first order theory as $G$ (by the L\"owenheim-Skolem Theorem such a group exists).  Then $H$ also has torsion, and is therefore not n-, cm-, or lcH-slender.

Suppose now that $G$ is nontrivial and torsion-free.  Let $\Phi$ denote the first order theory of $G$ (including the conditions which define a group).  We extend the signature of groups to include infinitely many constants $\{c_m\}_{m\in \omega}$ and let $\Theta$ denote the set of formulas $\{c_{m}=c_{m+1}^{m+1}\}_{m\in \omega} \cup \{c_0 \neq 1\}$.  Notice that any finite subset of $\Phi \cup \Theta$ is satisfiable.  More particularly, since $G$ is torsion-free and nontrivial we have $G$ as a model of $\Phi \cup F$ where $F$ is any finite subset of $\Theta$.  By model theoretic compactness there exists a model of $\Phi \cup \Theta$ and since $\Phi \cup \Theta$ has countable signature there exists countable structure $H$ which satisfies $\Phi \cup \Theta$ (by the L\"owenheim-Skolem Theorem).  Then $H$ is a countable group with the same first order theory as $G$ which contains a subgroup $\langle \{c_m\}_{m\in \omega}\rangle$ which is isomorphic to $\mathbb{Q}$ (isomorphism follows from the fact that $\langle \{c_m\}_{m\in \omega}\rangle$ is a nontrivial group and $H$ is torsion-free).  Then $H$ is not n-, cm-, or lcH-slender.
\end{proof}

\end{section}

\end{document}